\documentclass[reqno]{amsart}

\usepackage{amsfonts,color,newclude}
\usepackage{amsthm}
\usepackage{amssymb,transparent}
\usepackage{amsmath}
\usepackage{amscd}
\usepackage{thm-restate}

\usepackage{mathabx}
\usepackage{graphicx}
\usepackage{subfiles}
\usepackage{tikz}
\usepackage{tikz-cd}
\usetikzlibrary{calc}
\usepackage{MnSymbol}
\usepackage[colorlinks=true, pdfstartview=FitV, linkcolor=blue, citecolor=blue, urlcolor=blue]{hyperref}
\usepackage[normalem]{ulem} 
\usepackage{hyperref}

\usepackage[nameinlink,capitalize]{cleveref}

\renewcommand{\tilde}{\widetilde}
\newcommand{\mathsym}[1]{{}}

\newcommand{\naturals}{\mathbb{N}}

\newcommand{\integers}{\mathbb{Z}}

\newcommand{\inv}{^{-1}}

\newcommand{\Z}{\integers}

\newcommand{\oskel}{^{(1)}}

\newcommand{\Stab}{\operatorname{Stab }}

\def\mc {\mathcal}

\def\CAT {\ensuremath{\operatorname{CAT}}}

\newcommand{\boundary}{\partial}

\newcommand{\hyp}{\mathfrak{h}}		

\newcommand{\leftQ}[2]{\left.\raisebox{-.2em}{$#2$}\middle\backslash\raisebox{.2em}{$#1$}\right.}
\newcounter{probnum}
\newtheorem{theorem}{Theorem}[section]

\newtheorem{proposition}[theorem]{Proposition}

\newtheorem{lemma}[theorem]{Lemma}
\newtheorem{example}[theorem]{Example}

\theoremstyle{definition}
\newtheorem{definition}[theorem]{Definition}

\renewcommand{\tilde}{\widetilde}

\renewcommand{\epsilon}{\varepsilon}
\newcommand{\walls}{\mathcal{H}_{\mathrm{0}}}

\title{On the boundary criterion for relative cubulation: multi-ended parabolics}

\author{Eduard Einstein}
\author{Suraj Krishna M S}
\author{Thomas Ng}

\begin{document}
	\maketitle
	
	\setcounter{tocdepth}{1}
	
	\begin{abstract}
		In this note, we clarify that
		the boundary criterion for relative cubulation of the first author and Groves works even when the peripheral subgroups are not one-ended.
		Specifically, if the boundary criterion is satisfied for a relatively hyperbolic group, we show that, up to taking a \emph{refined peripheral structure}, the group
		admits a relatively geometric action on a CAT(0) cube complex.  
		We anticipate that this refinement will be useful for constructing new relative cubulations in a variety of settings.
	\end{abstract}

	\section{Introduction}
	
	In \cite{Sageev95}, Sageev introduced a powerful technique for constructing a CAT(0) cube complex from a given family of codimension-1 subgroups of a group $G$. 
	Since then, the study of groups acting properly and cocompactly on hyperbolic CAT(0) cube complexes and the program studying virtually special cube complexes initiated by Haglund and Wise \cite{HW08} have produced a number of important results in geometric group theory, including Agol's resolution of the Virtual Haken and Virtual Fibering Conjectures \cite{AgolVirtualHaken}.
	
	
	There is substantial interest in using cubical machinery to produce more general results, specifically for relatively hyperbolic groups, see for example \cite{BergeronWise,HruskaWise14,PiotrWise18,SW2015,WiseManuscript}. 
	In all of these works, the focus is on studying/obtaining either proper and cocompact, or proper but not cocompact actions of relatively hyperbolic groups on CAT(0) cube complexes.
	In \cite{RelCannon}, 
	the first author and Groves introduced relatively geometric actions, a type of cocompact  but improper action of a 
	relatively hyperbolic group
	on a CAT(0) cube complex.
	In subsequent work \cite{RelGeom}, they prove that groups admitting relatively geometric actions have properties similar to those of hyperbolic groups that act geometrically on CAT(0) cube complexes. For example, a relatively hyperbolic group with residually finite parabolics is residually finite if it acts relatively geometrically on a CAT(0) cube complex \cite[Corollary 1.7]{RelGeom}. 
	
	\begin{definition}
		\label{def:relGeom}
		Let $(G,\mc{P})$ be a relatively hyperbolic pair, and let $\tilde{X}$ be a CAT(0) cube complex. A cellular action of $G$ on $\tilde{X}$ is \emph{relatively geometric} (with respect to $\mc{P}$) if
		\begin{enumerate}
			\item $\leftQ{\tilde{X}}{G}$ is compact,
			\item every $P\in\mc{P}$ acts elliptically,
			\item every infinite cell stabilizer is conjugate to a finite index subgroup of some $P\in\mc{P}$. 
		\end{enumerate}
	\end{definition}
	
	Bergeron and Wise leveraged Sageev's construction and
	developed a boundary criterion 
	to obtain
	proper and cocompact actions of hyperbolic groups on CAT(0) cube complexes \cite{BergeronWise}.  
	The first author and Groves adapted the machinery of \cite{BergeronWise} to produce a relatively geometric version of the boundary criterion, \cite[Theorem 2.6]{RelCannon}, for relatively hyperbolic groups with one-ended peripherals. In this note, we explain how to extend \cite[Theorem 2.6]{RelCannon} to a relatively geometric cubulation criterion for arbitrary peripherals. However, we may need to alter the peripheral structure in a natural way: 
	
	\begin{definition}
		Let $(G,\mc{P})$ be a relatively hyperbolic pair. If there is a collection of subgroups $Q$ such that each $Q\in\mc{Q}$ is conjugate into some $P\in \mc{P}$ and $(G,\mc{Q})$ is relatively hyperbolic, then we say that $(G,\mc{Q})$ is a \emph{refined peripheral structure} for $(G,\mc{P})$. 
	\end{definition}
	
	Properties of the peripheral structure $\mc{P}$ that can be deduced from $\mc{Q}$ have been studied by Yang \cite{WenyuanYang2014}, who calls $\mc{P}$ a parabolically extended structure for $\mc{Q}$.
	Our main result is the following:
	
	\begin{restatable}{theorem}{mainthm}    
		\label{thm: boundary multiended}
		Let $(G, \mc P)$ be a finitely generated relatively hyperbolic group 
		that satisfies the boundary criterion in \cite{RelCannon}, that is, there exists a collection $\mc H$ of codimension-1 subgroups of $G$ satisfying the following conditions.
		\begin{enumerate}
			\item 
			\label{item:full}
			Every $H \in \mc H$ is $\mc P$-full.
			\item 
			\label{item:relQC}
			Every $H \in \mc H$ is $\mc P$-relatively quasiconvex.
			\item 
			\label{item:separation}
			(Separation)
			For every pair of distinct points $\alpha$ and $\beta$ in $\boundary_{\mc P} G$, there exists a subgroup $H \in \mc H$ such that the limit set $\Lambda H$ of $H$ separates $\alpha$ and $\beta$, i.e., $\alpha$ and $\beta$ lie in $H$-distinct components of $\boundary_{\mc P} G \setminus \Lambda H$.
		\end{enumerate}
		Then there is a finite subcollection $\mathcal{H}_0 = \{H_1, \cdots, H_n\}$ of $\mc H$ such that $G$ acts cocompactly on the dual CAT(0) cube complex $X$. If each $P\in\mc{P}$ acts elliptically on $X$, then $(G,\mc{P})$ is relatively geometric. Otherwise, there is a refinement $\mc Q$ of $\mc P$ so that $(G,\mc{Q})$ acts relatively geometrically on $X$.    
	\end{restatable}
	The proof of \Cref{thm: boundary multiended} follows the same outline as the proofs of the cubulation criteria in \cite{BergeronWise,RelCannon}. We find $\walls$ in the same way as \cite{BergeronWise}. In the setting of \cite[Theorem 2.6]{RelCannon}, where every $P\in\mc{P}$ is one-ended, we can immediately guarantee that each $P\in\mc{P}$ acts elliptically on the CAT(0) cube complex $X$ dual to $\walls$. However, in our setting,  $P\in\mc{P}$ can intersect a conjugate of an element of $\walls$ in a finite codimension--$1$  subgroup. This may prevent $P$ from acting elliptically on $X$. When the $P\in\mc{P}$ do not act elliptically, we also cannot immediately conclude that $G$ acts cocompactly on $X$. 
	
	To resolve these issues, we find a refined relatively hyperbolic structure where the maximal peripherals act elliptically. 
	We need to carefully refine the peripherals so that the elements that act loxodromically on $X$ are exactly the hyperbolic elements in the new refined structure for $G$. 
	If $P\in\mc{P}$ fails to act elliptically on $X$, then $P$ has a finite codimension--$1$ subgroup $F=H^g\cap P$ where $H\in\walls$ and $g\in G$. We initially expected that classical results of Dunwoody \cite{dunwoody} and Stallings \cite{StallingsEndsTF,StallingsBookYale} would induce an iterated splitting of $P$ so that the resulting vertex groups of the splitting act elliptically on $X$ and the elements of $P$ that do not lie in a vertex group of the splitting act by hyperbolic isometries on $X$. 
	However, \cite[Theorem 4.1]{dunwoody} only allows us to take an initial splitting of $P$ over a subgroup of $P$ commensurable to $F$, which gives very little information when $F$ is finite. 

Instead, we prove the following more general result that 
may be of independent interest:
\begin{restatable}{theorem}{finitecodim}  \label{T: the finer structure}
	Let $P$ be a finitely generated group and let $K_1,\ldots,K_n$ be a collection of finite codimension--$1$ subgroups of $P$. 
	Let $\tilde{Y}$ denote the dual CAT(0) cube complex. Then $\tilde{Y}$ is a quasi-tree. Moreover, $P$ is virtually free or hyperbolic relative to the collection of conjugacy representatives of the infinite vertex stabilizers of $\tilde{Y}$.  
\end{restatable}
\cref{T: the finer structure} is particularly useful because the loxodromic elements of $P$ relative to the (conjugacy classes of) vertex stabilizers of $\tilde{Y}$ will act on $\tilde{Y}$ by hyperbolic isometries while the parabolic subgroups are contained in vertex stabilizers of $\tilde{Y}$. 
That $\tilde{Y}$ is a quasi-tree is straightforward and follows from \cite{manning_bottleneck}. We additionally prove that $P$ acts cocompactly on $\tilde{Y}$ and that the one-skeleton of $\tilde{Y}$ is a fine hyperbolic graph so that the action of $P$ on $\tilde{Y}\oskel$ satisfies \cite[Definition 3.4 (RH-4)]{Hruska2010}. 

In the case of \Cref{thm: boundary multiended}, we show in \Cref{Sec: endgame} that the elements of a $P\in\mc{P}$ that act as hyperbolic isometries on $\tilde{Y}$ from \Cref{T: the finer structure} act as hyperbolic isometries of $X$. Similarly, the vertex stabilizers of $\tilde{Y}$ act elliptically on $X$. 
In light of \cite[Theorem 2.6]{RelCannon}, one might ask whether the refined peripheral structure that we obtain for \Cref{thm: boundary multiended} by using \Cref{T: the finer structure} must have one-ended peripherals. The answer is no, see Example~\ref{E: multi ended refined peripherals} below. 

We expect that \Cref{thm: boundary multiended} will be useful for constructing new relatively geometric cubulations. 
The one-ended version, \cite[Theorem 2.6]{RelCannon}, of \Cref{thm: boundary multiended} is already useful for producing relative cubulations of finite volume hyperbolic $3$-manifold groups and producing relative cubulations of certain free-product-by-cyclic groups \cite{DahmaniKrishnaMS}. Settings where we anticipate using  \Cref{thm: boundary multiended} include small cancellation free products of relatively cubulable groups \cite{C16:relCube}, random quotients of free products 
in forthcoming work of the authors with  Montee and Steenbock, 
and proving a relatively geometric generalization of Hsu and Wise's work \cite{HW2015} on cubulating malnormal amalgams of cubulated relatively hyperbolic groups by the first author and Groves. 
In the hyperbolic case, \cref{thm: boundary multiended}, coupled with Theorem D of \cite{grovesmanningimproper} can in fact be used to obtain proper and cocompact actions of hyperbolic groups on CAT(0) cube complexes --- as was done recently in \cite{hypbycycliccube}.


\subsection{Acknowledgements}
The authors thank Daniel Groves, Nir Lazarovich, Jason Manning, Michah Sageev, and Markus Steenbock for helpful conversations and comments that helped shape this work. 
The first author received support from an AMS-Simons research fellowship that funded travel to collaborate with the second and third author on this paper.
The second author was supported by ISF grant 1226/19 at the Technion, and by the International Centre for Theoretical Sciences during the program Geometry in Groups (Code: ICTS/GIG2024/07).
The third author was partially supported by 
ISF grant 660/20 and 
at the Technion by a Zuckerman Fellowship.

\section{Properties of refined peripheral structures}

Relatively quasiconvex subgroups of refined peripheral structures are closely related.
In \cite{WenyuanYang2014}, $\mc{P}$ is called a parabolically extended structure for $\mc{Q}$. 
Rephrasing \cite[Theorem~1.1]{WenyuanYang2014}, we obtain:
\begin{proposition}\label{P: old parabolics rel qc}
	Let $(G,\mc{P})$ be relatively hyperbolic and let $(G,\mc{Q})$ be any refined peripheral structure. For all $P\in\mc{P}$ and all $g\in G$, $P^g$ is relatively quasiconvex in $(G,\mc{Q})$. 
\end{proposition}

Recall that infinite order elements of $(G,\mc{P})$ can be classified as either hyperbolic or parabolic, based on the action of $(G, \mc{P})$ on its Bowditch boundary. See, for instance, \cite{Hruska2010} for details.

A subgroup $H$ is \emph{$\mc P$-full} if for every $P \in \mc P$ and $g\in G$, the intersection $H \cap P^g$ is either finite or a finite index subgroup of $P^g$. 

\begin{proposition}\label{prop: full and qc preserved}
	Suppose that $(G,\mc{P})$ is a relatively hyperbolic pair and $(G,\mc{Q})$ is any refined peripheral structure. If $H\le G$ is $\mc{P}$-full, then $H$ is $\mc{Q}$-full. 
	If in addition, $H$ is $\mc{P}$-relatively quasiconvex then $H$ is $\mc{Q}$-relatively quasiconvex.
\end{proposition}
\begin{proof}
	Fullness is straightforward.
	Each $P^g$ is $\mc{Q}$-relatively quasiconvex by \cref{P: old parabolics rel qc}.
	Since $H$ is $\mc{P}$-full, either $H \cap P^g$ is finite or finite index in $P^g$.  
	Therefore, by \cite[Theorem~1.3]{WenyuanYang2014} $H$ is $\mc{Q}$-relatively quasiconvex.
\end{proof}

\section{Some facts about cubical actions}


In this section, we introduce a 
few
results that will be useful for working with actions on CAT(0) cube complex.

Groups that act by isometries on a finite dimensional CAT(0) cube complex are semisimple in the following sense:
\begin{proposition}
	[{\cite[Proposition~B.8]{CapraceChatterjiFernosIozzi}}]
	\label{prop:semisimple}
	Let $X$ be a finite dimensional $\CAT(0)$ cube complex. 
	Let $K \leq Isom(X)$ be a finitely generated subgroup with unbounded orbit. 
	There exists $g \in K$ and hyperplane $\hyp \subset X$ with halfspace $\hyp_+$ such that $g\hyp_+ \subsetneq \hyp_+$.
\end{proposition}

The following condition is useful for identifying which elements of a group acting on a CAT(0) cube complex act loxodromically:
\begin{definition}[{\cite[Above Lemma~5.4]{BergeronWise}}]
	\label{def:axis-sep} 
	Let $\mc{H}$ be a collection of codimension--$1$ subgroups of $G$. 
	An infinite order element $h\in G$ satisfies the \textbf{axis separation condition} for $\mc{H}$ if there exists $H \in \mc{H}$, $n \in \mathbb{N}$ and $g \in G$ such that some choice of halfspace $H^g_+$ satisfies
	\[
	h^{-n}H^g_+ \subsetneq H^g_+ \subsetneq h^n H^g_+.
	\]
\end{definition}

\begin{definition}\label{D: skewered}
	A hyperplane $\hyp$ in a CAT(0) cube complex is \textbf{skewered} by the hyperbolic isometry $h$ if there is a choice of halfspace $\hyp^+$ and exponent $n$ such that $h^n\hyp^+ \subsetneq \hyp^+$.  
\end{definition}

The following behavior is straightforward to see and is pointed out by Bergeron and Wise \cite{BergeronWise}.  

\begin{lemma}
	\label{L: skewering separates axes}
	Suppose that $(G, \mc{Q})$ is a relatively hyperbolic pair and that $G$ acts on a CAT(0) cube complex. Let $\hyp$ be a hyperplane with relatively quasiconvex stabilizer $H$. If $\hyp$ is skewered by a loxodromic element $h \in G$ then 
	$h$ satisfies axis separation for any collection of subgroups containing $\Stab(\hyp)$.  
\end{lemma}
\begin{proof}[Proof sketch]
	By relative quasiconvexity of $H = \Stab(\hyp)$ the limit set $\Lambda_\mc{Q}H$ is a closed proper subset of $\boundary_\mc{Q}G$ disjoint from the conical limit points fixed by $h$.  A sufficiently large power of $h$ will move $\Lambda_\mc{Q}(H)$ completely off of itself 
	by north-south dynamics of loxodromic elements on the Bowditch boundary, and the result follows. 
\end{proof}

The following appears in the proof of Lemma 5.5 of \cite{BergeronWise}:

\begin{lemma}
	\label{lem:stabilizer}
	Let $(G, \mc{Q})$ be a relatively hyperbolic group
	and $\mc{H}_0$ be a finite collection of codimension-1 subgroups so that every hyperbolic element of $(G,\mc{Q})$ satisfies the axis separation condition for $\walls$. 
	Then every elliptic subgroup for the action on the cube complex dual to $\mc{H}_0$ is either finite or contained in some parabolic subgroup $Q^g$ for $Q \in \mc{Q}$ and $g \in G$. 
\end{lemma}

When the maximal parabolics act elliptically, we obtain a cocompact action. 

\begin{lemma}\label{L: cocompact follows}
	Let $(G,\mc{Q})$ be a relatively hyperbolic group. Let $\mathcal{H}_0$ be a finite collection of full relatively quasiconvex codimension-1 subgroups and let $X$ be the dual CAT(0) cube complex. If each $Q \in \mathcal{Q}$ acts elliptically on $X$, the action of $G$ on $X$ is cocompact.  
\end{lemma}

\cref{L: cocompact follows} follows from Proposition 2.5 of \cite{RelCannon}; we give a sketch here for completeness. Since each subgroup $H \in \mathcal{H}_0$ is relatively quasiconvex, by \cite[Theorem 7.12]{HruskaWise14}, the action of $G$ on $X$ is \emph{relatively cocompact}, i.e, there exists a compact subcomplex $K$ and for each $Q \in \mathcal{Q}$ a  convex subcomplex $C_{Q}$ stabilised by $Q$ such that $$X = GK \cup \bigcup_{Q \in \mathcal{Q}} GC_Q,$$
with $gC_Q \cap g'C_{Q'}$ being either equal to $gC_Q$, or contained in $GK$.
By the assumption that each $Q$ acts elliptically, $C_Q$ is compact and thus $X$ is $G$-cocompact. 

The last result in this section is about the commensurability of cube stabilizers with maximal parabolic subgroups. First, we will need the following auxiliary lemma: 

\begin{lemma}\label{L: cat0 to combinatorial}
	Let $X$ be a CAT(0) cube complex, and let $G$ be a group acting combinatorially by isometries on $X$. If $K\le G$ stabilizes vertices $v_1$ and $v_2$, then 
	there exists a combinatorial geodesic $\sigma$ between $v_1$ and $v_2$ and a finite index $K_0 \le K$ such that $K_0$ stabilizes $\sigma$ pointwise. 
\end{lemma}

\begin{proof}
	Let $\gamma$ be the CAT(0) geodesic between $v_1$ and $v_2$. 
	Then there exists a finite sequence of cubes $C_1,C_2,\ldots,C_k$ so that $\gamma$ intersects the interior of $C_i$ and $\gamma = \bigcup_{i=1}^k (\gamma\cap C_i)$. By indexing following the path of $\gamma$ from $v_1$ to $v_2$, we may arrange for $C_i\cap C_{i+1}\ne\emptyset$. 
	
	Observe that $K$ stabilizes $\gamma$ pointwise because $X$ is CAT(0) and therefore uniquely geodesic. 
	Since $K$ stabilizes a point in the interior of each $C_i$, a finite index subgroup $K_i\le K$ stabilizes each vertex of $C_i$ and therefore stabilizes $C_i$ pointwise.
	If $K_0 = \bigcap_{i=1}^k K_i$, $K_0$ stabilizes every $C_i$ pointwise. 
	Since $C_i\cap C_{i+1}\ne\emptyset,$ $C_i$ and $C_{i+1}$ share a common cube, their intersection contains a vertex. 
	Thus there is a combinatorial path in $\bigcup_{i=1}^k C_i$ from $v_1$ to $v_2$, which must be stabilized pointwise by $K_0$. 
\end{proof}

\begin{lemma}\label{P:stabilizers commensurable max parabolic}
	Let $(G,\mc{Q})$ be a relatively hyperbolic pair so that $G$ acts on a CAT(0) cube complex $X$ cocompactly by isometries with the following properties:
	\begin{enumerate}
		\item every maximal parabolic acts elliptically,
		\item every infinite cell stabilizer is parabolic, and
		\item every hyperplane stabilizer is full.
	\end{enumerate}
	Then every infinite cell stablizer is finite index in a maximal parabolic subgroup. 
\end{lemma}

\begin{proof}
	Let $K$ be an infinite stabilizer of a cell $c$ and let $P_0$ be the maximal parabolic subgroup that contains $K$. We will show $K$ is finite index in $P_0$. We may assume that $c$ is a vertex because a finite index subgroup of $K$ stabilizes every vertex of $c$. 
	Since $P_0$ acts elliptically, there is a vertex $v$ stabilized by $P_0$. 
	
	By \cref{L: cat0 to combinatorial}, there is a finite index $K_0\le K$ so that $K_0$ stabilizes a combinatorial path $\sigma$ from $v$ to $c$ pointwise. 
	Let $e$ be the edge incident to $v$ in $\sigma$ (stabilized by the infinite subgroup $K_0$). 
	Let $\hyp$ be the hyperplane dual to $e$. Note that $P_{\hyp} = \Stab_G(\hyp) \cap P_0$ is finite index in $P_0$ as hyperplane stabilizers are full. Also, since $P_{\hyp}$ fixes a vertex of $e$, it has to fix $e$ as hyperplanes of CAT(0) cube complexes do not self osculate. The proof now follows by repeating the argument on each edge of $\sigma$ until we reach $c$.
	%
	%
	%
	Thus $c$ has stabilizer which is finite index in $P_0$. 
\end{proof}

\section{Finite codimension-1 subgroups witness relative hyperbolicity}

In this section, we prove \cref{T: the finer structure}, which we restate here for the reader's convenience. 

\finitecodim*

\begin{proof}
	There are three subclaims which we prove as propositions and then use these to prove \Cref{T: the finer structure} using Bowditch's characterization of relative hyperbolicity in terms of actions on fine hyperbolic graphs \cite{BowditchRH,Hruska2010}. 
	\begin{proposition}
		$\tilde{Y}$ is a quasi-tree. 
	\end{proposition} 
	
	\begin{proof}
		Every hyperplane of $\tilde{Y}$ is compact because each $K_i$ is a finite group. Since any path between two points has to cross every (finite) hyperplane that a geodesic between the two points crosses, Manning's bottleneck criterion \cite{manning_bottleneck} is satisfied by $\tilde{Y}$. 
	\end{proof}
	
	\begin{proposition} 
		$P$ acts cocompactly on $\tilde{Y}$ (with finitely many hyperplane orbits)
	\end{proposition}
	
	\begin{proof}
		The compactness of each hyperplane also implies that $\tilde{Y}$ is finite-dimensional, since any hyperplane can cross at most finitely many hyperplanes. Observe that this also means that there are finitely many orbits of edges: any edge is dual to a hyperplane (which is compact and thus dual to finitely many edges), and there are finitely many orbits of hyperplanes. 
	\end{proof}

	\begin{proposition}
		The one skeleton $\tilde{Y}^{(1)}$ of $\tilde{Y}$ is fine in the sense of Bowditch. 
	\end{proposition}
	
	\begin{proof}
		Let $r\in \Z^{>0}$, and let $e$ be a fixed edge of $\tilde{Y}^{(1)}$. 
		Since $\tilde{Y}$ is hyperbolic, $\tilde{Y}^{(2)}$ satisfies a linear isoperimetric inequality.
		Since there are only finitely many orbits of $2$--cells, there are finitely many orbits of disk diagrams of perimeter $r$ for $\tilde{Y}$.
		A circuit of length $r$ that contains $e$ encloses a disk diagram. There are only finitely many such disk diagrams, so there can be only finitely many circuits of length $r$ that contain $e$. 
	\end{proof}
	
	Since $\tilde{Y}^{(1)}$ is a fine hyperbolic graph with finite edge stabilizers and finitely many edge orbits, $\tilde{Y}^{(1)}$ satisfies \cite[Definition 3.4 (RH-4)]{Hruska2010} and $P$ is therefore relatively hyperbolic relative to the infinite vertex stabilizers. 
	Note that in the case that all vertex stabilizers are finite then the action is also proper, so $P$ is virtually free.
\end{proof}

\section{Refining to an elliptic peripheral structure}\label{Sec: endgame}

Our main technical result is \Cref{T: refined action} which explains how we can use \Cref{T: the finer structure} to refine the peripherals of $G$ so that hyperbolic elements in the refined peripheral structure act loxodromically on the dual cube complex $X$ and the refined peripherals will act elliptically on $X$. 
\begin{theorem}\label{T: refined action}
	Let $(G,\mc{P})$ be relatively hyperbolic and let $X$ be the cube complex dual to a finite collection $\walls$ of full relatively quasiconvex (with respect to $(G,\mc{P})$) subgroups. 
	Then each $P\in\mc{P}$ is either virtually free or hyperbolic relative to a collection $\mc{Q}_P$ of subgroups so that
	\begin{enumerate}
		\item each $Q\in \mc{Q}_P$ acts elliptically on $X$ and
		\item if $g\in P$ is hyperbolic in $(G,\mc{Q})$, then $g$ satisfies axis separation for $\walls$. 
	\end{enumerate}
\end{theorem}

If each $P$ acts elliptically on $X$, \Cref{T: refined action} follows immediately, so for the remainder of this section, we assume that a given $P$ does \emph{not} act elliptically on $X$. 

We now introduce some notation:
\begin{definition}
	Let $K\le G$. A \textbf{$K$--wall} of $G$ is a pair of $K$--invariant subsets $K^+,K^-\subseteq G$ so that $G = K^+\cup K^-$, $K^+,K^-$ are unions of infinitely many right cosets of $K$ and $K^+\cap K^-$ is a union of finitely many right cosets of $K$.
\end{definition}

\begin{lemma}
	\label{claim:finite-dimensional}
	Under the hypotheses of \cref{T: refined action}, the cube complex $X$ is finite dimensional.
\end{lemma}
\begin{proof}
	By construction, $X$ has finitely many $G$-orbits of hyperplanes.  
	A subgroup $A \leq B$ is said to have the \emph{bounded packing property} in $B$ if for every $d > 0$ there exists a $k$ such that any collection of $k+1$ distinct left cosets of $A$ contains a pair whose $d$-neighborhoods are disjoint where distance is measured in the Cayley graph of $B$ with respect to a finite generating set.  
	Hruska and Wise \cite[Theorem~3.30]{HruskaWise14} and 
	\cite[Theorem~8.10]{HruskaWise:packing} 
	prove that 
	if for each peripheral subgroup $P \in \mc{P}$ the intersection with any hyperplane stabilizer has the bounded packing property in $P$ then $X$ is finite dimensional.

	Finite subgroups clearly have the bounded packing property because the Cayley graph is locally finite, so $k$ may be taken to be the size of the $d$-ball in $G$.
	
	Hypothesis (\ref{item:full}) gives that each hyperplane stabilizer $H_i^g$ is full with respect to the peripheral structure $\mc P$. 
	Hence, if $|H_i^g\cap P| < \infty$, we have already seen that they have bounded packing.
	Otherwise, $H_i^g\cap P$ is finite index in $P$, but then bounded packing holds by taking $k = [P : H_i^g\cap P]$.
	\qedhere
\end{proof}

The cube complex $X$ is dual to a collection of $H_i$--walls $H_i^+,H_i^-$ in $G$.
We say that $P$ is \textbf{deep} in $gH_i^{\pm}$ if $\leftQ{gH_i^\pm\cap P}{H_i^g\cap P}$ is infinite.
We say that $gH_i$ \textbf{cuts $P$ essentially} if $P$ is deep in both $gH_i^+$ and $gH_i^-$. 
Since $H_i$ is full, $H_i^g\cap P$ is either finite index in $P$ or finite. If $H_i^g\cap P$ is finite index in $P$, the orbit of $gH$ under $P$ is coarsely equal to $gH$. Therefore, $H_i^g\cap P$ is finite whenever $H_i^g$ cuts $P$ essentially.

\begin{lemma}
	\label{claim:codim one intersections}
	If $P \in \mc{P}$ has unbounded orbits in $X$ then
	$P$ intersects a hyperplane stabilizer $H$ in a finite codimension-1 subgroup of $P$. In other words, if $P$ does not act elliptically on $X$, there is at least one $H_i^g$ that cuts $P$ essentially.
\end{lemma}
\begin{proof}
	Since $X$ is finite dimensional by \Cref{claim:finite-dimensional},  \Cref{prop:semisimple} guarantees that there exists an infinite order element $g \in P$ and halfspace $\hyp_+$ such that $g\hyp_+ \subsetneq \hyp_+$. 
	In particular, the orbit of $P$ escapes any finite neighbourhood of $\hyp$.
	Let $H = \Stab_G(\hyp)$ be the hyperplane stabiliser.  
	By the fullness assumption, $H \cap P$ is a finite subgroup of $P$ and translates by powers of $g$ witness that $H\cap P$ has two deep complementary components. 
	Indeed, if $H \cap P$ were finite index in $P$ then there would be only finitely many $P$-translates of $\hyp$ which we have already seen to be false.
\end{proof}

\begin{lemma}
	\label{L: essential walls exist}
	Assume that every element of $\mc{P}$ is finitely generated. 
	There exists $R>0$ so that whenever $gH_i$ cuts $P\in\mc{P}$ essentially, there is a $p\in P$ so that $d(p\inv gH_i,1_G)<R$. 
\end{lemma}

\begin{proof} 
	Fix a generating set $S_P$ for each $P\in \mc{P}$. 
	There are finitely many elements in $\mc{P}$, so there exists a $k>0$ so that $k$ exceeds the length in $G$ of any element of $\bigcup_{P\in\mc{P}}S_P$. 
	Choose $R_0$ so that:
	\[R_0> \max_i (\#\text{ of right cosets of $H_i$ in } H_i^+\cap H_i^-).\]
	We can view $P$ as a connected subgraph $\Gamma_P$ of the Cayley graph of $G$ where any point in $|\Gamma_P|$ is at most distance $k$ from a vertex of $\Gamma_P$ representing an honest element of $P$.  
	For $P$ to be deep in both of $gH_i^\pm$, there is some $x\in \Gamma_P$ so that $d(x, gH_i)<R_0$ because $gH_i^+\cap gH_i^-$ must cut $\Gamma_P$ into multiple deep components. Hence $d(P,gH_i)<R_0+k$. Letting $p\in P$ be a point that realizes this inequality, we then use the left invariance of the metric  
	to obtain $d(1_G,p\inv gH_i)<R_0+k$. Hence $R =R_0+k$ suffices.
\end{proof}

In light of \Cref{L: essential walls exist}, we conclude that there are finitely many $P$-orbits of $gH_i$ that cut $P$ essentially because there are finitely many left cosets that intersect the $R$-neighborhood of the identity. Let $\mc{H}_P$ be a (finite) collection of $P$--orbit representatives of $gH_i$ that cut $P$ essentially. By \Cref{claim:codim one intersections}, $\mc{H}_P$ is non-empty. 

Let $\tilde{Y}$ be the cube complex dual to the finite collection of walls
of $P$ corresponding to $gH_i\in \mc{H}_P$.  By \cref{T: the finer structure}, $P$ is hyperbolic relative to a (finite) collection $\mc{Q}_P$ of $P$--conjugacy representatives for infinite vertex stabilizers of $\tilde{Y}$. 

\begin{lemma}
	\label{L: lox X implies Y}
	Suppose that $h \in P$ acts hyperbolically on $X$.  
	Then $h$ acts hyperbolically on $\tilde{Y}$.
\end{lemma}
\begin{proof}
	An isometry $h$ of $X$ is hyperbolic if and only if it skewers some hyperplane \cite[Lemma~2.4]{CapraceSageev}, and any skewered hyperplane determines a disjoint \emph{skewer set} $sk(h) = \{ h^{nk}\hyp \mid k \in \Z \}$  (see for example Kar and Sageev \cite{KarSageev-UEG}).
	We see from this that $P$ is deep with respect to any hyperplane skewered by $h$, so each hyperplane $\frak{h} \in sk(h)$ corresponds to a coset in $P\mc{H}_P$.
	Indeed, disjoint hyperplanes are at distance at least 1 and separating.
	By construction, if two hyperplanes of $P\mc{H}_P$ have nested halfspaces in $X$, then the corresponding halfspaces are nested in $\tilde{Y}$. 
	Thus, $h$ is hyperbolic on $\tilde{Y}$.
\end{proof}

\begin{proposition}\label{P: new refined loxo}
	Let $h\in P$ act as a hyperbolic isometry on $\tilde{Y}$. Then $h$ satisfies the axis separation condition for $\walls$ and therefore acts hyperbolically on $X$.
\end{proposition}

\begin{proof}
	
	We will show that $h$ is skewered by a hyperplane of $X$.  
	Since $h$ is a hyperbolic isometry of $\tilde{Y}$, there exists a hyperplane $\hyp_Y$
	and $n \in \naturals$ such that $h^n\hyp_Y^+ \subsetneq \hyp_Y^+$ in $\tilde{Y}$.  Recall that the hyperplanes of $\tilde{Y}$ have finite stabilizer. 
	On the other hand, $h$ has infinite order, so $h^m\notin \Stab(\hyp_Y)$ for all $m\in \Z\setminus \{0\}$. 
	
	The hyperplane $\hyp_Y$ corresponds to a wall $Z = \{ Z^-, Z^+ \}$ in $P$. 
	By construction, $Z^\pm = P \cap W^\pm$ where $W = \{W^-, W^+\}$ is a wall in $G$. 
	Thus, nontrivial powers of $h$ do not stabilize $W$.
	Moreover, the nesting above implies that $P \cap h^{nm}W^+ \subseteq (W^+ \cap h^{nm}W^+)$ for all $m \in \naturals$.  
	While $W$ and $h^nW$ may cross, non-zero powers of $h$ do not lie in $\Stab_G(W)$. Thus $W,h^nW,h^{2n}W,\ldots$ are distinct cosets in $G\walls$ which correspond to distinct hyperplanes $\hyp_0,\hyp_1,\hyp_2,\ldots$ of $X$. Then there exists $ 1 \leq k \leq \dim(X)$ such that $\hyp_0$ is disjoint from $h^{nk}\hyp_0 = \hyp_n$.  Hence, $\hyp_n^+ \subsetneq \hyp^+$, so $\hyp$ skewers $h$.
	
	Axis separation follows from \Cref{L: skewering separates axes}.
\end{proof}

\begin{proposition}\label{P: new refined elliptics}
	Let $Q$ be an infinite vertex stabilizer of $\tilde{Y}$. Then $Q$ acts elliptically on $X$. 
\end{proposition}

\begin{proof}
	Suppose toward a contradiction that $Q$ is not elliptic on $X$.  
	Since $X$ is finite dimensional by \Cref{claim:finite-dimensional}, \Cref{prop:semisimple} guarantees that some element $q \in Q$ acts as a hyperbolic isometry of $X$. 
	By \Cref{L: lox X implies Y}, $q$ must also act as a hyperbolic isometry on $\tilde{Y}$.
	However, hyperbolic isometries do not stabilize vertices of $\tilde{Y}$, so we obtain a contradiction.
\end{proof}

\begin{proof}[Proof of \Cref{T: refined action}]
	Take $\mc{Q}_P$ to be a finite collection of $P$--conjugacy representatives for infinite vertex stabilizers of $\tilde{Y}$ as above. 
	If no such stabilizers exist then $P$ acts geometrically on $\tilde{Y}$ and is virtually free.
	As a set, the Bowditch boundary of $(P,\mc{Q}_P)$ is equal to the disjoint union of $\partial{Y}$, the visual boundary of $\tilde{Y}$ and the vertices of $\tilde{Y}$ with infinite stabilizer, and the parabolics are infinite stabilizers of vertices in $\tilde{Y}$, see \cite[Section 9]{BowditchRH} for details. Thus any $h\in P$ which is loxodromic in $(P,\mc{Q}_P)$ must also act by a hyperbolic isometry on $\tilde{Y}$. Then \Cref{P: new refined loxo} implies that $h$ satisfies the axis separation condition for $\walls$. 
	If $Q$ is an infinite vertex stabilizer, then \Cref{P: new refined elliptics} implies $Q$ acts elliptically on $X$. 
\end{proof}

If the peripherals were one-ended, then there was no need for refinement because the peripherals already act elliptically. In the following example, we show that the refinement given by \Cref{T: refined action} may produce peripherals that are not one-ended: 

\begin{example}\label{E: multi ended refined peripherals}
	Suppose that $A, B, C$ are finitely generated groups. 
	Let $G = A * B * C$.  
	Observe that $\mc{P} = \{ A*B, C\}$ gives a peripheral structure for $G$. 
	Take $K_1 = K_2$ to be the trivial subgroup. 
	Codimension-1 subgroups can have multiple choices of halfspaces.  
	Choose halfspaces (in $G$) of $K_1$ such that $K_1^+$ is the set of elements of $G$ whose normal form has first syllable in $A$.
	Choose halfspaces of $K_2$ such that $K_2^+$ is the set of elements of $G$ whose normal form has first syllable in $A * B$
	By construction the cube complex constructed from translates of these choices of halfspaces has vertex stabilizers conjugate to one of $A$, $B$, or $C$, so $\mc{Q} = \{A, B, C\}$ would be the refined peripheral structure given by \Cref{T: refined action}.

	This example does not detect whether any of $A,B,C$ are one-ended or split over finite subgroups nor if they are finitely presented.
	Also note that taking $K_1 = K_2$ was for convenience.  One could consider an amalgam of $A*B$ with $C$ over a finite edge group to construct a similar example where $K_1,K_2$ are distinct.
\end{example}

\section{The proof of \texorpdfstring{\cref{thm: boundary multiended}}{Theorem 1.3}}

We are now ready to prove \Cref{thm: boundary multiended} according to the following scheme.
{We will refine our peripheral collection $\mc{P}$ to a collection $\mc{Q}$ using \Cref{T: refined action} such that each element in the collection acts elliptically on $X$, obtaining a cocompact action by $G$. We will then show that every infinite cube stabilizer is parabolic. Then every infinite cell stabilizer will be finite index in a maximal parabolic subgroup, completing the proof.}
We recall the statement here for convenience.
\mainthm*

\begin{proof}[Proof of \Cref{thm: boundary multiended}]
	Given hypothesis (\ref{item:separation}) of the theorem, the same proof as that of \cite[Theorem 5.1]{BergeronWise} gives a finite subcollection \[
	\mathcal{H}_{\mathrm{0}} = \{ H_1, \dots H_k \} \subseteq \mathcal{H},
	\] 
	such that every infinite order hyperbolic element of $(G,\mathcal{P})$ acts loxodromically on the dual CAT(0) cube complex $X$.

	

	At this point, if each $P\in\mc{P}$ acts elliptically on $X$, then we can proceed without refining the peripherals. 
	By \Cref{T: refined action}, for each $P\in\mc{P}$ that does \emph{not} act elliptically on $X$, there is a relatively hyperbolic structure $(P,\mc{Q}_P)$ so that each $Q\in \mc{Q}_P$ acts elliptically on $X$ and every $g\in P$ that is hyperbolic in $(P,\mc{Q}_P)$ acts loxodromically on $X$. 
	We make a new relatively hyperbolic structure $(G,\mc{Q})$ by replacing every $P$ that does not act elliptically by the elements of $\mc{Q}_P$. Note that \Cref{T: refined action} now ensures that every maximal parabolic acts elliptically. 
	\Cref{L: cocompact follows} now implies that the action of $G$ on $X$ is cocompact. 
	
	Every hyperbolic element $g$ of $(G,\mc{Q})$ is either hyperbolic in $(G,\mc{P})$ or is hyperbolic in $(P,\mc{Q}_P)$ for some $P\in\mc{P}$. In the first case, $g$ acts loxodromically on $X$ by \cite[Theorem 5.1]{BergeronWise} as noted above. 
	In the second case, \Cref{T: refined action} implies $g$ acts loxodromically on $X$. 
	\Cref{lem:stabilizer} now implies that every infinite elliptic subgroup for the action of $G$ on $X$ must be parabolic in $(G,\mc{Q})$. 
	Then, \Cref{P:stabilizers commensurable max parabolic} implies that every infinite cell stabilizer is finite index in a maximal parabolic subgroup. 
	
	Therefore, the action of $(G,\mc{Q})$ on $X$ is relatively geometric. 
\end{proof}

\bibliography{cubes}
\bibliographystyle{alpha}
\end{document}